\documentclass[11pt]{amsart}

\usepackage{amsmath,amssymb}
\usepackage{accents}
\usepackage{graphicx}
\usepackage{verbatim}

\newlength{\dhatheight}

\newtheorem{theorem}{Theorem}[section]

\newtheorem{lemma}[theorem]{Lemma}
\newtheorem{remark}[theorem]{Remark}
\newtheorem{corollary}[theorem]{Corollary}
\newtheorem{definition}[theorem]{Definition}

\newcommand{\N}{{\mathbb N}}

\newcommand{\Z}{{\mathbb Z}}

\newcommand{\ra}{\rightarrow}
\newcommand{\ras}{{\stackrel{~~*}{\ra}}}

\newcommand{\ew}{\lambda}  
\newcommand{\gi}{1}  
\newcommand{\nf}{{\mathcal{N}}}

\newcommand{\ga}{\Gamma}

\newcommand\vece{{\vec E}}
\newcommand\vecp{{\vec P}}
\newcommand{\ega}{e_{g,a}}

\newcommand{\sym}{inverse-closed}

\newcommand{\ff}{\Phi} 
\newcommand{\sff}{\phi} 
\newcommand{\ttt}{{\mathcal T}} 
\newcommand{\alg}{graph(\sff)} %

\newcommand{\astkbl}{{algorithmically stackable}}

\newcommand{\wff}{flow function}
\newcommand{\wstf}{stacking function}

\newcommand{\lbl}{{\mathsf{label}}}  
\newcommand{\rep}{\rho}
\newcommand{\path}{{\mathsf{path}}}  

\newcommand{\sr}{synchronously regular}
\newcommand{\pad}{\mu}

\newcommand{\prs}{prefix-rewriting system}
\newcommand{\cpr}{complete prefix-rewriting system}

\newcommand{\crs}{complete rewriting system}

\newcommand{\la}{\Lambda}
\newcommand{\rr}{\pi}
\newcommand{\last}{{\mathsf{last}}}
\newcommand{\cll}{{\mathcal{L}}}
\newcommand{\redl}{{\mathsf{dcl}}}  
\newcommand{\dol}{{\mathbb{I}}}  
\newcommand{\cgt}{\succ}   

\newcommand{\suf}{\mathsf{suf}}
\newcommand*{\medcup}{\mathbin{\scalebox{1.5}{\ensuremath{\cup}}}}
\newcommand{\es}{\mathsf{expsum}}

\newcommand{\hb}{{\widehat{B}}}
\newcommand{\hq}{{\widehat{\nf_Q}}}

\begin{document}
\title[Homology and closure properties of autostackable groups]
{Homology and closure properties of autostackable groups}

\author[M.~Brittenham]{Mark Brittenham}
\address{Department of Mathematics\\
        University of Nebraska\\
         Lincoln NE 68588-0130, USA}
\email{mbrittenham2@math.unl.edu}

\author[S.~Hermiller]{Susan Hermiller}
\address{Department of Mathematics\\
        University of Nebraska\\
         Lincoln NE 68588-0130, USA}
\email{smh@math.unl.edu}

\author[A.~Johnson]{Ashley Johnson}
\address{Department of Mathematics\\
        University of North Alabama\\
         Florence AL, 35632, USA}
\email{ajohnson18@una.edu}


\begin{abstract}
Autostackability for finitely presented groups is
a topological property of the Cayley graph combined
with formal language theoretic restrictions, that
implies solvability of the word problem.  
The class of autostackable groups is known to
include all
asynchronously automatic groups with respect
to a prefix-closed normal form set, and all groups admitting
finite complete rewriting systems.  Although groups
in the latter two classes all satisfy the homological
finiteness condition $FP_\infty$, we show  that
the class of autostackable groups
includes a group that is not of type $FP_3$.
We also show that the class of autostackable groups
is closed under graph products and extensions.
\end{abstract}

\maketitle


\section{Introduction}\label{sec:intro}


Autostackable groups are
an extension of the notions of automatic
groups and groups with finite complete
rewriting systems, introduced by Holt and
the first two authors in~\cite{AS}.
An autostackable structure for a finitely generated
group implies a finite presentation,
a solution to the word problem, 
a recursive algorithm for building van Kampen
diagrams, and 
tame combability~\cite{britherm},~\cite{brithermtame}.
Moreover, in contrast to automatic groups,
the class of autostackable groups includes
all fundamental groups of 3-manifolds with 
a uniform geometry~\cite{AS}.

Autostackability is 
a topological property of the Cayley 
graph, together with a language theoretic
restriction on this property.  More specifically,
let $G$ be a group with a finite
inverse-closed generating set $A$, and   
let $\ga=\ga(G,A)$ be the associated
Cayley graph.  Denote the set of directed edges
in $\ga$ by $\vece$, and the set of directed
edge paths by $\vecp$.
For each $g \in G$ and $a \in A$,
let $\ega$ denote the directed edge 
with initial vertex $g$, terminal vertex $ga$,
and label $a$;
we view the two directed edges $e_{g,a}$ 
and $e_{ga,a^{-1}}$  to
have a single underlying undirected edge in $\ga$.

A {\em flow function}
associated to a maximal tree $\ttt$ in $\ga$ is a
function $\ff:\vece \ra \vecp$ 
satisfying the properties that: 
\begin{itemize}
\item[(F1)] For each edge $e \in \vece$,
the path $\ff(e)$ has the same initial and terminal
vertices as $e$.
\item[(F2d)] If the undirected edge underlying $e$ 
lies in the tree $\ttt$, then $\ff(e)=e$.
\item[(F2r)]  The transitive closure 
$<_\ff$ of the relation $<$ on
$\vec E$ defined by 
\begin{itemize}
\item[]
$e' < e$ whenever $e'$ lies on the path $\ff(e)$
and the undirected edges underlying both
$e$ and $e'$ do not lie in $\ttt$,
\end{itemize}
is a well-founded strict
partial ordering.
\end{itemize}
The flow function is {\em bounded} if there is
a constant $k$ such that for all $e \in \vec E$,
the path $\ff(e)$ has length at most $k$.
That is, the map $\ff$ fixes the edges lying in the tree $T$
and describes a ``flow'' of the
non-tree edges toward the tree (or toward the basepoint);
starting from a non-tree edge and
iterating this function finitely many times results
in a path in the tree.

In order to place a language theoretic restriction
on $\ff$, we use functions that convert between
paths and words. Define $\lbl:\vecp \ra A^*$ to be 
the function that maps each directed path
to the word labeling that path.
For each element $g \in G$, let $y_g$ denote the
label of the unique geodesic (i.e., without backtracking)
path in the maximal tree $\ttt$ from the identity element
$\gi$ of $G$ to $g$, and let
$\nf=\nf_\ttt:=\{y_g \mid g \in G\}$ denote the 
set of these (unique) normal forms.
Define $\path:\nf \times A^* \ra \vecp$ by
$\path(y_g,w) :=$ the path in $\ga$ that 
starts at $g$ and is labeled by $w$.

\begin{definition}\cite{britherm,AS}\label{def:autostackable}
Let $G$ be a group with a finite inverse-closed generating
set $A$.
\begin{enumerate}
\item The group $G$ is {\em stackable} over $A$ if there is a bounded
flow function on a maximal tree in the associated
Cayley graph.  
\item The group $G$ is {\em algorithmically stackable}
over $A$ if $G$ admits a bounded flow function $\ff$ for which 
the graph 
$$
\alg:=\{(y_g,a,\lbl(\ff(\path(y_g,a)))) \mid g \in G, a \in A\}
$$
of the {\em stacking map} 
$\sff:=\lbl \circ \ff \circ \path$ is
computable.
\item The group $G$ is {\em autostackable}
over $A$ if $G$ has a bounded flow function $\ff$
for which the graph 
of the associated stacking map is  \sr.
\end{enumerate}
\end{definition}

A stackable group $G$ over a finite generating
set $A$ is finitely presented, with finite
presentation 
$R_\ff = \langle A \mid 
\{\sff(y,a) = a \mid y \in \nf_\ttt, a \in A\}\rangle$
associated to the \wff\ $\ff$.  The set $\nf_\ttt$
is a prefix-closed set of normal forms for $G$.
A bounded flow function is equivalent to
a bounded complete prefix-rewriting system
for $G$ over $A$, for which the irreducible
words are exactly the elements of the set
$\nf_\ttt$.  
(See Section~\ref{subsec:rs} for
definitions of rewriting and prefix-rewriting
systems.) 
Moreover, a group is
autostackable if and only if it admits
a \sr\ bounded complete \prs. 
Algorithmic stackability (and hence
also autostackability) implies a solution of the
word problem; the set of rules
of the associated prefix-rewriting system are
computable, and give an algorithm to
rewrite any word to the normal form representing
the same group element. 
The class of autostackable groups includes
all groups that are asynchronously automatic
with respect to a prefix-closed set of
(unique) normal forms, and all groups that 
admit a finite complete rewriting system.
The class of stackable groups also includes
all almost convex groups.
For proofs of these and other results
on autostackable groups, 
see~\cite{britherm} and~\cite{AS}.

Section~\ref{sec:notation} of this paper 
contains notation and definitions
used throughout the paper, 
including background on language theory.

In Section~\ref{sec:closure}, we show that the
classes of autostackable, stackable, and \astkbl\ groups
are all closed under taking graph products (including
free and direct products), extensions, and
finite index supergroups 
(i.e., groups containing a finite index subgroup in
the class).  For the two properties that 
motivated autostackability, we note that the
class of groups admitting a finite complete rewriting
system is closed under all three of these 
constructions~(\cite{GP},~\cite{hermeierartin},~\cite{grovessmith}),
but the class of automatic groups is only closed under
graph products and finite index supergroups~(\cite{GP},~\cite{bgss});
in particular, a nilpotent group  that is not virtually
abelian is not automatic~\cite[Theorem~8.2.8]{echlpt}.
The closure results in Section~\ref{sec:closure} show that any
extension of a automatic group by another automatic
group, such that the normal forms in both
cases are prefix-closed (and unique), is autostackable.

In Section~\ref{sec:stallings}
we show that the class of autostackable groups
includes groups with a wider range of homological
finiteness properties than those of automatic groups
or groups with finite complete rewriting systems.
A group $G$ has homological type $FP_n$ if there 
is a partial projective resolution of length $n$, by finitely
generated $\Z G$-modules, of
the module $\Z$ (with trivial $G$ action).
In the case that $G$ has type $FP_n$ for all
$n \in \N$, then $G$ is said to be of type $FP_\infty$.
Alonso~\cite{alonso} has shown that all groups
that admit a bounded combing, including all automatic
groups, have type $FP_\infty$.
Groups with finite complete rewriting systems also 
are of type $FP_\infty$; this has been shown with
a variety of proofs in papers by
Anick~\cite{Anick}, Brown~\cite{BrownRS}, Groves~\cite{GrovesRS}, 
Farkas~\cite{Farkas}, Kobayashi~\cite{Kobayashi}, and 
Lafont~\cite{Lafont};
see Cohen's survey~\cite{CohenSurvey} for more details.

Stallings~\cite{Stallings} showed that the group
$$
G := \langle a,b,c,d,s \mid [a,c] = [a,d] = [b,c] = [b,d] = 1, 
[s, ab^{-1}] = [s, ac^{-1}] = [s, ad^{-1}] = 1\rangle
$$
does not have the finiteness property $FP_3$.
The results above show that this group cannot be
automatic, nor can it admit a finite complete
rewriting system.  
Moreover, Elder and the second author have
shown that this group does not satisfy the almost convex
property~\cite{elder}, nor the weaker 
minimally almost convex property~\cite{elderherm}, on 
this generating set.
However, in Section~\ref{sec:stallings},
we show in Theorem~\ref{thm:stallings}
that this group is autostackable. 

\smallskip

\noindent{\bf Corollary~\ref{cor:autostknotfp3}.} {\em
There is an autostackable group that does not satisfy
the homological finiteness condition $FP_3$.  }

\smallskip

\noindent Stallings' group also provides an
example of a group that cannot have a finite complete
rewriting system, but does
admit a bounded complete prefix-rewriting system.


\section{Notation and background} \label{sec:notation}


Throughout this paper, let $G$ be a group
with a finite inverse-closed generating set $A$.
Also throughout the paper we assume that no element of a
generating set 
represents the identity element of the group,
and no two elements of a generating set
represent the same element of the group.

A set $\nf$ of {\em normal forms} for $G$ over $A$ is a 
subset of 
$A^*$ such that the restriction of the
canonical surjection $\rep: A^* \ra G$
to $\nf$ is a bijection.
As in Section~\ref{sec:intro}, the symbol $y_g$ denotes
the normal form for $g \in G$.  By slight abuse
of notation, we use the symbol $y_w$ to denote the
normal form for $\rep(w)$ whenever $w \in A^*$.

Let $\gi$ denote the identity of $G$, and let
$\ew$ denote the empty word in $A^*$.
For a word $w \in A^*$, we write $w^{-1}$ for the 
formal inverse of $w$ in $A^*$, and let $l(w)$ denote
the length of the word $w$.
For words $v,w \in A^*$, we write $v=w$ if $v$
and $w$ are the same word in $A^*$, and write $v=_G w$ if
$v$ and $w$ represent the same element of $G$.

Given a word $w \in A^*$, let $\last(w)$ denote
the last letter in $A$ of the word $w$; in the case that
$w = \ew$ contains no letters, then we let $\last(w) := \ew$.
For any subset $Z \subset A$, we use $\suf_Z(w)$,
to denote the maximal suffix 
of $w$ that lies
in $Z^*$; here $\suf_Z(w):=\ew$ 
if $w$ does not end with a letter in $Z$.

Let $\ga$ be the Cayley graph of $G$ with
respect to $A$.
If $\nf$ is a prefix-closed set of normal
forms for $G$ over $A$, then $\nf$ determines a
maximal tree $\ttt$ in $\ga$, namely the 
set of all (undirected) edges underlying
edge paths in $\ga$ starting at the vertex $\gi$
and labeled by words in $\nf$.


\subsection{Formal language theory} \label{subsec:language}


A {\em language} over a finite set $A$ is 
a subset of the set $A^*$ of all finite words over $A$.
The set $A^+$ denotes the language $A^* \setminus \{\ew\}$
of all nonempty words over $A$.

The {\em regular} languages over $A$ are the
subsets of $A^*$ obtained from the finite subsets
of $A^*$ using finitely many operations from among
union, intersection, complement, concatenation
($S \cdot T := \{vw \mid v \in S$ and $w \in T\}$),
and Kleene star ($S^0:=\{\ew\}$, $S^n := S^{n-1} \cdot S$ and
$S^* := \cup_{n=0}^\infty S^n$).  
The class of regular
languages is closed under both image and
preimage via monoid homomorphisms
(see, for example,~\cite[Theorem~3.5]{hu}).
The class of regular sets is also closed under quotients
(\cite[Theorem~3.6]{hu}); we write out a special case of 
this in the following lemma for use in later sections of this
paper.

\begin{lemma}~\cite[Theorem~3.6]{hu}\label{lem:peel}
If $A$ is a finite set, $L \subseteq A^*$ is a regular
language, and $w \in A^*$,
then the quotient language
$L/w := \{x \in A^* \mid xw \in L\}$ is also a 
regular language.
\end{lemma}

Let $\$$ be a symbol not contained in $A$.
The set $A_n:=(A \cup \{\$\})^n \setminus \{(\$,...,\$)\}$
is the {\em padded $n$-tuple alphabet} derived from $A$.
For any $n$-tuple of words $u=(u_1,...,u_n) \in (A^*)^n$,
write $u_i=a_{i,1} \cdots a_{i,j_i}$ with each 
$a_{i,m} \in A$ for $1 \le i \le n$ and $1 \le m \le j_i$.
Let $M:=\max\{j_1,...,j_n\}$, and define
$\tilde u_i:=u_i\$^{M-j_i}$, so that each of 
$\tilde u_1$, ..., $\tilde u_n$ 
has length $M$.  That is, $\tilde u_i$ is
a word over the alphabet $(A \cup \{\$\})^*$, and
we can write $\tilde u_i = c_{i,1} \cdots c_{i,M}$
with each $c_{i,m} \in A \cup \{\$\}$.
The word $\pad(u):=(c_{1,1},...,c_{n,1}) \cdots
(c_{1,M},...,c_{n,M})$ is the {\em padded word}
over the alphabet $A_n$ induced by the $n$-tuple 
$(u_1,...,u_n)$ in $(A^*)^n$.  

A subset $L \subseteq (A^*)^n$ is called
a {\em \sr\ language} if the {\em padded extension}
set $\pad(L) := \{\pad(u) \mid u \in L\}$ of
padded words associated to the elements of $L$ is
a regular language over the alphabet $A_n$.
Closure of the class of \sr\ languages under
finite unions and intersections follows from 
these closure properties for regular languages.
The following two lemmas on \sr\ languages 
will also be used in later sections.

\begin{lemma}~\cite[Lemma~2.3]{AS}\label{lem:product}
If $L_1,...,L_n$ are regular languages over $A$,
then their Cartesian product 
$L_1 \times \cdots \times L_n \subseteq (A^*)^n$
is \sr.
\end{lemma}

\begin{lemma}~\cite[Theorem~1.4.6]{echlpt}\label{lem:proj}
If $L \subset (A^*)^n$ is a
\sr\ language, then the projection on the first
coordinate given by the set 
${\mathsf{proj}}_1(L):=\{u \mid \exists (u,u_2,...,u_n) \in L\}$
is a regular language over $A$.
\end{lemma}

See~\cite{echlpt} and~\cite{hu} for more information
about regular and \sr\ languages.


\subsection{Rewriting systems} \label{subsec:rs}


The definitions and results in this section
can be found in the text~\cite{sims} by Sims.

A {\em \crs}\label{def:crs} for a group $G$ consists of a set $A$
and a set of ``rules'' $R \subseteq A^* \times A^*$
(with each $(u,v) \in R$ written $u \ra v$)
such that $G$ is presented 
as a monoid by 
$G = Mon\langle A \mid u=v$ whenever $u \ra v \in R \rangle,$
and the rewritings of the form
$xuy \ra xvy$ for all $x,y \in A^*$ and $u \ra v$ in $R$,
with transitive closure $\ras$, satisfy:
\begin{enumerate}
\item[(1)] There is no infinite chain $w \ra x_1 \ra x_2 \ra \cdots$
of rewritings.
\item[(2)] Whenever there is a pair of rules of the form
$rs \ra v$ and $st \ra w$ [respectively, $s \ra v$ and $rst \ra w$]
in $R$ with $r,s,t,v,w \in A^*$ and
$s \neq \ew$, then there are rewritings $vt \ras z$ and $rw \ras z$
[respectively, $rvt \ras z$ and $w \ras z$]
for some $z \in A^*$.
\end{enumerate}
The rewriting system is {\em finite} if the sets 
$A$ and $R$ are both finite.

The pairs of rules in item (2) are 
called {\em critical pairs}, and when property
(2) holds, the critical pairs are said to be
{\em resolved}.  
The set $Irr(R)$ of irreducible words
(that is, words that cannot be rewritten) is a
set of normal forms for the group $G$ presented
by the complete rewriting system.

A {\em \cpr} for a group $G$ consists of a set $A$
and a set of rules $R \subseteq A^* \times A^*$
(with each $(u,v) \in R$ written $u \ra v$)
such that  $G$ is presented
(as a monoid) by 
$G = Mon\langle A \mid u=v$ whenever $u \ra v \in R \rangle,$
and the rewritings 
$uy \ra vy$ for all $y \in A^*$ and $u \ra v$ in $R$ satisfy:
{(1)} There is no infinite chain $w \ra x_1 \ra x_2 \ra \cdots$
of rewritings, and
{(2)} each $g \in G$ is 
represented by exactly one irreducible word
over $A$.
(The difference between a \prs\ and a rewriting system 
is that rewritings of the form $xuy \ra xvy$ with
$x \in A^* \setminus \{\ew\}$ and $u \ra v \in R$ 
are allowed in a rewriting system, but only
rewritings $uy \ra vy$ are allowed in a \prs.)
The \prs\ is {\em bounded} if $A$ is finite
and there is a constant
$k$ such that for each pair $(u,v)$ in $R$
there are words $s,t,w \in A^*$ such that
$u=ws$, $v=wt$, and $l(s)+l(t) \le k$.


\section{Closure properties of autostackable groups}\label{sec:closure}



\subsection{Graph products}\label{subsec:gp}


$~$

\vspace{.1in}

In this section we prove the first of
the closure properties, that each of
the stackability properties is preserved
by the graph product construction.

Given a finite simplicial graph $\la$
(with no loops or multiple edges) 
with vertices $v_1,...,v_n$, such that
each vertex $v_i$ is labeled by a group $G_i$,
the associated {\em graph product} is
the quotient $G\la$ of the free product
of the groups $G_i$ by the relations
that elements of vertex groups corresponding
to adjacent vertices in $\la$ commute.
Special cases include the free product 
(if $\la$ is totally disconnected) and direct
product (if $\la$ is complete) of the groups $G_i$.

For each $1 \le i \le n$, let $A_i$ be a finite \sym\ 
generating set for the vertex group $G_i$.
In this section we use the generating set 
$A := \cup_{i=1}^n A_i$ of $G\la$ for our constructions.
For each $i$, we let $I_i \subseteq \{1,...,n\}$
denote the set of indices $k$ such that
$v_k$ and $v_i$ are adjacent in $\la$.
This set $I_i$ can be partitioned into the
subsets $I_i^{>} := I_i \cap \{i+1,...,n\}$
and $I_i^{<} := I_i \cap \{1,...,i-1\}$.
Let $C_i:=A_i \cup \{\dol,\cgt\}$, 
where $\dol$ and $\cgt$ denote distinct letters not in $A$,
and define a monoid homomorphism
$\rr_i:A^* \rightarrow C_i^*$ by
defining 
\[\rr_i(a) := \begin{cases}
a & \mbox{if } a \in A_i \\
\cgt& \mbox{if } a \in A_k \text{ for some } k \in I_i^{>} \\
\ew & \mbox{if } a \in A_k \text{ for some } k \in I_i^{<} \\
\dol& \mbox{if } a \in A_k \text{ for some } k \in 
                              I \setminus (I_i \cup \{i\}).
\end{cases}\]

\begin{lemma}\label{lem:nfgrprod}
Let $G\la$ be a graph product of the groups $G_i=\langle A_i \rangle$,
let $A:=\cup_{i=1}^n A_i$, and suppose that
for each index $i$ the
set $\nf_i$ is a prefix-closed set of normal forms 
for $G_i$ over the generators $A_i$.
Then the language
$$
\nf_\la:=\cap_{i=1}^n \rr_i^{-1}((\nf_i\cgt^*\dol)^*\nf_i\cgt^*)
$$
is a prefix-closed set of normal forms for $G\la$.
\end{lemma}

\begin{proof}
Over the larger generating set 
$X:=\cup_{i=1}^n X_i$ of $G\la$ 
where each $X_i := G_i \setminus \{\gi_{G_i}\}$,
we note that the set of rules
$
R:=R' \cup R'',
$ 
where
$
R' :=
\{gh \ra (gh) \mid g,h \in X_i,~i \in \{1,...,n\}\}
$
and
$
R'' := \{gwh \ra hgw \mid g \in X_i, 
  h \in X_j, j \in I_i^{<}, w \in (\cup_{k \in I_j} X_k)^*\},
$
is a \crs\ for $G$.
Here $(gh)$ denotes the element of $X_i$ corresponding
to the product $gh$ in $G_i$ if $gh \neq_{G_i} \gi_{G_i}$,
and $(gh)$ denotes the empty word $\ew$ if $gh=_{G_i} \gi_{G_i}$.
Indeed,
if we let $S:=\{s_1,...,s_n\}$ 
have the total ordering defined by $s_i<s_j$ whenever $i<j$,
and define the monoid homomorphism 
$\alpha: X^* \ra S^*$ by $\alpha(g):=s_i$ for each $g \in X_i$,
then each rewriting $xuy \ra xvy$ with $x,y \in X^*$ and $u \ra v \in R$
satisfies the property that $\alpha(xuy) >_{sl} \alpha(xvy)$,
where $>_{sl}$ is the (well-founded) shortlex ordering on $S^*$,
and so there cannot
be an infinite sequence of rewritings.  
It is also straightforward to check that the critical
pairs are resolved (see Section~\ref{subsec:rs} for
this terminology), and so this is a \crs.
Hence the set $Irr(R)$ of irreducible words
for this system is a set of normal forms
for $G\la$ over $X$.

Now let $\beta:X^* \ra A^*$ be the monoid homomorphism 
mapping each $g \in X_i$ to the normal
form $\beta(g)$ of $g$ in $\nf_i$.
Then $\beta(Irr(R))$ is a set of normal forms for $G\la$
over $A$.  

Given a word $w \in Irr(R)$,
the image $\rr_i(\beta(w))$  lies in 
$(N_i\cgt^*\dol)^*N_i\cgt^*$ for all $i$,
by the choice of the
rewriting rules in $R$; hence,
$\beta(Irr(R)) \subseteq \nf_\la$.  
In the other direction, for any word 
$x \in \nf_\la$, we can consider the word
$x$ as an element of $X^*$ using the inclusion of
$A$ in $X$.  Since $\rr_i(x)$ lies in
$(\nf_i\cgt^*\dol)^*\nf_i\cgt^*$ for each $i$,
 the only rules of the rewriting system
$R$ that can be applied to $x$ are in the set $R'$.
Since each element of $G_i$ is represented
by only one word in $\nf_i$, and
the normal form set $\nf_i$ is prefix-closed,
it follows that nonempty subwords of words in
$\nf_i$ cannot represent the trivial element
of $G_i$.  Consequently any sequence of rewritings
of $x$ using the rules in $R'$
may only replace words in $X_i^+$ (that is,
nonempty words over $X_i$) with
words again in $X_i^+$.  
Hence any further rewritings from the system $R$
again can only apply rules in $R'$, resulting in
an irreducible word $x'$.  Applying $\beta$ returns the
original word $\beta(x')=x$.  
Therefore $\nf_\la = \beta(Irr(R))$.

Finally, prefix-closure of the sets $\nf_i$
yields prefix-closure of the languages
$(\nf_i\cgt^*\dol)^*\nf_i\cgt^*$ for each $i$,
which in turn implies prefix-closure of 
the language 
$\nf_\la$.
\end{proof}

We note that the normal forms $Irr(R)$ in Lemma~\ref{lem:nfgrprod}
are the same as those developed by Green in~\cite{Green},
and the set $\nf_\la$ is
also constructed using alternate methods in~\cite{GP} and~\cite{HsuWise}.
Next we use the normal form set $\nf_\la$ to
prove the closure properties for graph products.

\begin{theorem}\label{thm:grprod}
For $1 \le i \le n$ let $G_i$ be an autostackable
[respectively, stackable, \astkbl] group on
a finite \sym\ generating set $A_i$.
Then any graph product 
$G\la$ of these groups with the
generating set 
$A:=\cup_{i=1}^n A_i$ is also autostackable
[respectively, stackable, \astkbl].
\end{theorem}

\begin{proof}
Let $\nf_i$, $\ff_i$, and $\sff_i$ be the normal form set over $A_i$,
the bounded flow function, and 
the stacking map for the group $G_i$, respectively.
To streamline the discussion,
for each $1 \le i \le n$ we denote the language
$(\nf_i\cgt^*\dol)^*\nf_i\cgt^*$ by $\cll_i$.
Also let $\nf_\la$
be the normal form set for $G\la$ from Lemma~\ref{lem:nfgrprod},
and as usual denote the normal form for $g \in G\la$ by $y_g$.
Let $\ga$ be the Cayley graph of $G\la$ over $A$, with
sets $\vec E$ of directed edges and $\vec P$ of
directed edge paths, and let
$\ttt$ be the maximal tree in $\ga$ 
corresponding to this set of
normal forms.

\smallskip

\noindent {\em Step 1:  Stackable.}

We begin by defining a function $\sff:\nf_\la \times A \ra A^*$
as follows.  
Recall from Section~\ref{sec:notation} that for
any word $w \in A^*$, 
$\last(w)$ denotes the last letter of the word $w$,
and
$\suf_{A_i}(w)$, which we shorten to 
$\suf_i(w)$ throughout this proof, denotes the maximal suffix 
of $w$ in the letters of the subset $A_i$ of $A$.
Now for each $y_g \in \nf_\la$ and $a \in A_k$ we define
\[\sff(y_g,a) := \begin{cases}
\sff_k(\suf_k(y_g),a) & 
       \mbox{if } \rr_k(y_g) \notin C_k^*\cgt \\
\last(y_g)^{-1}a\last(y_g) & 
       \mbox{if } \rr_k(y_g) \in C_k^*\cgt.
\end{cases}\]
We also let $\ff:\vece \ra \vecp$ denote the function
$\ff(\ega):=\path(g,\sff(y_g,a))$.

It follows immediately from the definition of $\ff$ that
property (F1) of the definition of flow function holds.

To check property (F2d), consider any directed edge $e=\ega$ 
whose underlying undirected edge lies in the tree $\ttt$.
Then either $y_ga$ or $y_{ga}a^{-1}$ is an element
of $\nf_\la$.  Let $k$ be the index such that 
$a \in A_k$.  Now either
$\rr_k(y_ga)=[\rr_k(y_g)]a \in [(\nf_k\cgt^*\dol)^*\nf_k\cgt^*]a
\subseteq \cll_k$
or
$\rr_k(y_{ga}a^{-1})= \rr_g(y_g) \in \cll_k a^{-1} \cap \cll_k$,
and so in both cases we have
$\rr_k(y_g) \notin C_k^*\cgt$.  
Hence $\sff(y_g,a)=\sff_k(\suf_k(y_g),a)$.
Write $\rr_k(y_g)=vw$ where $v \in (\nf_k\cgt^*\dol)^*$
and $w \in \nf_k$.  Since $\rr_k(b)=\ew$
for all $b \in \cup_{i \in I_k^{<}} A_i$,
then in the word $y_g$, 
the letters from $w$ may be interspersed with,
or precede, such a letter $b$.  However, 
if $w \neq \ew$ and 
$b$ is the first letter in $\cup_{i \in I_k^{<}} A_i$ 
occurring after the 
first letter of $w$ in $y_g$, and we let
$a'$ be the letter from $w$ that immediately
precedes $b$ in $y_g$, then
for the index $i$ such that $b \in A_i$, the word
$\pi_i(a'b)=~\cgt b$ is a subword of $\rr_i(y_g)$,
giving a contradiction.  This shows that
$\suf_k(y_g)=w$.
Now we have that either $y_ga \in \nf_\la$,
in which case $\suf_k(y_g)a=wa \in \nf_k$,
or else $y_{g}=y_{ga}a^{-1}$, in which case
$\suf_k(y_g)$ ends with the letter $a^{-1}$.
Since the flow function
$\ff_k$ satisfies property (F2d),
then $\sff_k(\suf_k(y_g),a)=a$.  Therefore
$\ff(e)=e$, and (F2d) holds for $\ff$.

Next we turn to property (F2r).
Define a function $\psi:\vece \ra \N^2$ 
on $\ega \in \vece$ with $a \in A_k$ by
$\psi(\ega):=(l(y_g),0)$ if  $\rr_k(y_g) \in (C_k)^*\cgt$, 
and
$\psi(\ega):=(0,\redl_k(\suf_k(y_g),a))$ if
$\rr_k(y_g) \notin (C_k)^*\cgt$,
where the
{\em descending chain length}
$\redl_k(w,a)$ denotes the maximum length of a descending chain 
$e_{w,a} >_{\ff_k} e_1 >_{\ff_k} e_2 \cdots >_{\ff_k} e_n$
of edges for the well-founded ordering obtained
from the flow function $\ff_k$.
Let $<_{\N^2}$ be the lexicographic ordering on $\N^2$;
that is, $(a,b) <_{\N^2} (c,d)$ if either $a<c$  or
$a=c$ and $b<d$, where we use the standard ordering on $\N$.
Note that $<_{\N^2}$ is a well-founded strict partial ordering.
In order to show that property (F2r) holds for 
the function $\ff$, it suffices to show that
whenever $e' <_{\ff} e$, then $\psi(e') <_{\N^2} \psi(e)$.
Making use of the fact that the ordering 
$<_\ff$ is a transitive closure of another relation, 
it then suffices to show that
whenever $e'$ is an edge of $\ff(e)$ and
neither $e'$ nor $e$ is in the tree $\ttt$,
then $\psi(e') <_{\N^2} \psi(e)$.

Now suppose that $e=\ega$ is any element
of $\vece$ and the undirected
edge of $\ga$ underlying $\ega$ does not
lie in $\ttt$, and let $k$ be the index
such that $a \in A_k$. 
In this case the words $y_g a$ and $y_{ga}a^{-1}$
are not in the normal form set $\nf_\la$.
Note that for any index $i \neq k$, the image of $y_g a$
under $\rr_i$ has the form $\rr_i(y_g a)=\rr_i(y_g)\rr_i(a)$
where $\rr_i(y_g) \in \cll_i$ and $\rr_i(a) \in \{\cgt,\ew,\dol\}$, 
and therefore also $\rr_i(y_g a) \in \cll_i$.
Then since $y_g a \notin \nf_\la$, we must have
$\rr_k(y_g a) =\rr_k(y_g)a \notin \cll_k=(\nf_k\cgt^*\dol)^*\nf_k\cgt^*$.
We treat the cases in which $\rr_k(y_g)$
does or does not end with the letter $\cgt$ separately.

Suppose first that 
$\rr_k(y_g)$ does not end with $\cgt$. 
Then $\rr_k(y_g) \in (\nf_k\cgt^*\dol)^*\nf_k$;
write $\rr_k(y_g) = uv$ for $u \in (\nf_k\cgt^*\dol)^*$
and $v \in \nf_k$.
Applying the discussion in the proof
of (F2d) above, then $v=\suf_k(y_g)$; hence
we can write $y_g=\tilde u v$.
Now $\psi(e)=(0,\redl_k(\suf_k(y_g),a))=(0,\redl_k(v,a))$
where $\redl_k(v,a)$
is the maximum length of a descending chain of edges
starting from $e_{v,a}$ for the flow function $\ff_k$,
and since $\rr_k(y_g a) \notin \cll_k$ and $\rr_k(y_g)$
does not end with $a^{-1}$, the descending
chain length satisfies $\redl_k(v,a)>0$.
For any edge $e'$ in the path $\ff(e)$,
we have $e'=e_{\tilde u v',c}$
for some $v' \in \nf_k$ and $c \in A_k$ such that
$e_{v',c}$ is an edge of the path 
$\ff_k(e_{v,a})$.  Note that,
as above, $\suf_k(\tilde u v')=v'$.  In the
case that $e'$ does not lie in the tree $\ttt$,
then $e_{v',c} <_{\ff_k} e_{v,a}$ and 
$\rr_k(\tilde u v')=\rr_k(\tilde u)v'
\in \cll_k \cap (C_k^* \setminus C_k^*\cgt)$,
and hence $\psi(e')=(0,\redl_k(\suf_k(y_{\tilde u v'}),c))
=(0,\redl_k(v',c))
<_{\N^2} (0,\redl_k(v,a))=\psi(e)$.

On the other hand suppose  that 
$\rr_k(y_g)$ ends with $\cgt$.    
The path $\ff(e)$ contains three directed edges
$e_1:=e_{g,b^{-1}}$, 
$e_2:=e_{gb^{-1},a}$, and
$e_3:=e_{gb^{-1}a,b}$ where $b:=\last(y_g)$.
Note that $y_g=y_{gb^{-1}}b$, and $l(y_g) \ge 1$.
The undirected edge underlying
$e_1$ lies in the tree $\ttt$.
If $e_2$ does not lie in the tree $\ttt$, then
$y_{gb^{-1}} a \notin \nf_\la$, and so
the argument (two paragraphs) above also shows that
$\rr_k(y_{gb^{-1}} a) \notin \cll_k$.  
In that case, either $\rr_k(y_{gb^{-1}})$
ends with the letter $\cgt$, in which case
$\psi(e_2)=(l(y_{gb^{-1}}),0)=(l(y_g)-1,0) <_{\N^2} (l(y_g),0)=\psi(e)$,
or else $\rr_k(y_{gb^{-1}} a)$ does not end with $\cgt$,
in which case
$\psi(e_2)=(0,\redl_k(\suf_k(y_{gb^{-1}}),a))<_{\N^2} (l(y_g),0)=\psi(e)$.
So in all cases, $\psi(e_2) <_{\N^2} \psi(e)$.

Finally, to analyze the edge $e_3$, we first
note that since $\rr_k(y_g) \in (\nf_k\cgt^*\dol)^*\nf_k\cgt^+$,
we can decompose the normal form for $g$
as $y_g = u_1u_2u_3b$ where 
$\rr_k(u_1) \in (\nf_k\cgt^*\dol)^*$ and
the length of the prefix $u_1$ is maximal with respect
to this property, $u_2 \in \nf_k$ (including the
possibility that $u_2$ is the empty word $\ew$),
$\rr_k(u_3b) \in~\cgt^+$, and the first letter $c$ of 
$u_3b$ satisfies $\rr_k(c) =~\cgt$.
Note that $gb^{-1}ab =_{G\la} ga =_{G\la} u_1y_{u_2a}u_3b$;
we claim that $u_1y_{u_2a}u_3b \in \nf_\la$.
Indeed, for the index $k$ we have 
$\rr_k(u_1y_{u_2a}u_3vb) = \rr_k(u_1)y_{u_2a}\rr_k(u_3b)
\in (\nf_k\cgt^*\dol)^*\nf_k \cgt^+ \subseteq \cll_k$.
For each index $i \neq k$, the word
$\rr_i(u_1y_{u_2a}u_3vb)$ is obtained from
$\rr_i(y_g)$ by removal of the subword 
$\rr_i(u_2)$ and insertion of the subword 
$\rr_i(y_{u_2a})$; we consider these steps separately. 
The word $\rr_i(u_1u_3b)$  
is obtained from $\rr_i(y_g)$ by removing 
a subword from 
$\cgt^* \cup~\dol^+$.  Deleting a $\cgt^*$
subword preserves membership in the language $\cll_i$.
In the case that
$\rr_i(u_2) \in \dol^+$, then the vertices 
$v_i$ and $v_k$ of $\la$ are not adjacent,
and $\rr_k(d)=\dol$
for all $d \in A_i$; hence no letter of
$u_3b$ lies in $A_i$.  Removal of a
$\dol^+$ subword from a suffix in
$\{\cgt,\dol\}^*$ also preserves membership
in $\cll_i$.  Hence we have
$\rr_i(u_1u_3b) \in \cll_i$.
Next, $\rr_i(u_1y_{u_2a}u_3b)$ 
is obtained
from $\rr_i(u_1u_3b)$ by the insertion of 
a subword from 
$\cgt^* \cup~\dol^+$, and
insertion of such a subword preserves
membership in $\cll_i$ unless the 
inserted word is nonempty and immediately precedes a
letter lying in $A_i$. 
In the case that $\rr_i(y_{u_2a}) \in~\cgt^+$,
we have $i<k$.  
Moreover, the first letter
$c$ of $u_3b$ satisfies
$\rr_k(c) =~\cgt$, and so $c \in A_l$
for an index $l$ satisfying
$k<l$.  Then $i<l$, and so the word $\rr_i(y_{u_2a})$
is inserted immediately before the letter
$\rr_i(c) \in \{\cgt,\dol\}$.
In the case that $\rr_i(y_{u_2a}) \in \dol^+$,
again the vertices $v_i$ and $v_k$ are
not adjacent and no letter of $A_i$ appears
in $u_3b$, so no $\dol$ is inserted
preceding a letter of $A_i$.  
Thus $\rr_i(u_1y_{u_2a}u_3b) \in \cll_i$,
 completing the proof of the claim.
Since the normal form $y_{gb^{-1}ab}=u_1y_{u_2a}u_3b$ 
ends with the letter $b$, then the
edge $e_3$ lies in $\ttt$.

Thus we have that for any edge $e'$ in the path $\ff(e)$,
either $e'$ lies in the tree $\ttt$, or else
$\psi(e') <_{\N^2} \psi(e)$, as required.
This completes the proof of (F2r) and
the proof that $\ff$ is a flow
function associated to the tree $\ttt$
in the Cayley graph for $G\la$ over $A$.
It follows from the construction
of $\ff$ that this flow function is bounded by the
constant $\max\{3,bd(\ff_1),...,bd(\ff_n)\}$,
where for each index $i$, $bd(\ff_i)$ denotes the bound on
the flow function $\ff_i$.

\smallskip

\noindent {\em Step 2:  Autostackable.}

Now suppose further that the set $graph(\sff_k)$
is \sr\ for all indices $1 \le k \le n$.
Following the piecewise description of the
\wstf\ $\sff$ associated to $\ff$ given
in Step 1 above, its graph can be written in the
form
\begin{eqnarray*}
graph(\sff) &=& \cup_{k=1}^n \cup_{a \in A_k} 
~[  (\cup_{x \in im(\sff_k)} 
       ~graph(\sff) \cap
             (A^* \times \{a\} \times \{x\})) \\
&& \medcup (\cup_{i \in I_k} \cup_{b \in A_i} 
       ~graph(\sff) \cap (A^* \times \{a\}
       \times \{b^{-1}ab\})) ]\\
&=& (\cup_{k \in \{1..n\},~a \in A_k,~x \in \{\sff_k(y,a) \mid y \in \nf_k\}} 
      \ \ \ \ L_{k,a,x} \times \{a\} \times \{x\}) \\
&& \medcup (\cup_{k \in \{1..n\},~a \in A_k,~i \in I_k,~b \in A_i} 
      \ \ \ \ L_{k,a,b}' \times \{a\} \times \{b^{-1}ab\} )
\end{eqnarray*}
where

\hspace{.1in} $L_{k,a,x} = \{y_g \in \nf_\la \mid \rr_k(y_g) \notin C_k^*\cgt$
    and $\sff_k(\suf_k(y_g),a)=x \}  $,~~and

\hspace{.1in} $L_{k,a,b}' = \{y_g \in \nf_\la \mid \rr_k(y_g) \in C_k^*\cgt$
    and $\last(y_g)=b\}      $

Since the class of \sr\ languages is closed under finite
unions,
and all finite sets are regular, then Lemma~\ref{lem:product}
shows that in order to prove that $graph(\sff)$ is \sr,
it suffices to show that the languages $L_{k,a,x}$ and $L_{k,a,b}'$
are regular.

We begin by considering the language $\nf_\la$ of normal forms.
Lemma~\ref{lem:proj} says that the projection of
$graph(\sff_k)$ on the first coordinate, which 
is the language $\nf_k$ of normal forms for the group 
$G_k$, is regular.  Since regular languages are
closed under concatenation and Kleene star,
then $\cll_k=(\nf_k\cgt^*\dol)^*\nf_k\cgt^*$ is also regular.
Finally closure of regular languages under 
homomorphic inverse image and finite intersection implies
that $\nf_\la$ is also regular.

Now $L_{k,a,b}' = \nf_\la \cap \rr_k^{-1}(C_k^*\cgt) \cap A^*b$.
Since the class of regular languages is closed under Kleene
star, concatenation, and homomorphic preimage, then
each of the three sets in this intersection is regular, and
so $L_{k,a,b}'$ is a regular language.

The proof in Step 1 above shows that
for any 
$y_g \in L_{k,a,x} \subseteq \nf_\la \cap \rr_k^{-1}(C_k^* \setminus C_k^*\cgt)$, 
the word $y_g$ can be written in the form $y_g=y'\suf_k(y_g)$ where
$\rr_k(y') \in (\nf_k\cgt^*\dol)^*$.
Then
$L_{k,a,x} = \nf_\la 
\cap 
\rr_k^{-1}((\nf_k\cgt^*\dol)^*) N_{k,a,x}$ where
$N_{k,a,x}:={\mathsf{proj}}_1(graph(\sff_k) \cap (A_k^* \times \{a\} \times \{x\}))$
is the set of all words
$w \in \nf_k$ such that $\sff_k(w,a)=x$.
Since $graph(\sff_k)$ is \sr\ and
the intersection of two \sr\ languages is \sr,
then applying Lemma~\ref{lem:proj} shows that the language $N_{k,a,x}$
is regular.  Then using closure properties of
regular languages, we also have that $L_{k,a,x}$ is regular.

Therefore the set 
$graph(\sff)$ is \sr, and $G\la$ is autostackable over $A$.

\smallskip

\noindent{\em Step 3:  Algorithmically stackable.}

The proof in this case is similar to the argument in Step 2.
\end{proof}


\subsection{Extensions}\label{sec:ext}


$~$

\vspace{.1in}

We continue the investigation into closure properties with the 
extension of a group $K$ by a group $Q$.

\begin{theorem}\label{thm:extn}
Let  
$
1   \rightarrow K \overset{\iota}{\rightarrow}G \overset{q}{\rightarrow} Q \rightarrow 1
$
be a short exact sequence of groups and group homomorphisms.  
If $K$ and $Q$ are 
autostackable
[respectively, stackable, \astkbl] groups on
finite \sym\ generating sets $A$ and $B$, respectively,
and $\hb \subseteq G$ is an inverse-closed subset
of $G$ that bijects via $q$ to $B$, then
the group $G$ with the generating set $i(A) \cup \hb$
is also autostackable
[respectively, stackable, \astkbl]. 
\end{theorem}

\begin{proof}
Let $\nf_K$, $\ff_K$, and $\sff_K$ be the normal form
set, bounded flow function, and associated stacking
map for $K$ over $A$, and similarly let 
$\nf_Q$, $\ff_Q$, and $\sff_Q$ be the normal form
set, bounded flow function, and associated stacking
map for $Q$ over $B$.
Let $K = \langle A | R \rangle$ and $Q = \langle B | S \rangle$
be the finite presentations obtained from 
these flow functions. 
By slight abuse of notation, we 
will consider the homomorphism $i$ to be an
inclusion map, and $A,K \subseteq G$,
so that we may omit writing $i(\cdot)$.
Let $C := A \cup \hat{B}$.  

For each $b \in B$, there is a unique element
$\hat b \in \hb$ with $q(\hat b)=b$.
For each word $w = b_1 \cdots b_n\in B^*$, we define 
$\hat{w}:= \hat{b_1} \cdots \hat{b_n}$. 
Let $\bar{\ }$ be the map
from $\hb^*$ to $B^*$ that reverses
the map $\hat{~}$; that is, for any letter $c \in \hb$,
$\overline{c}=q(c) \in B$, and for any
word $v=c_1 \cdots c_n \in \hb^*$, 
then $\overline{v}=\overline{c_1} \cdots \overline{c_n}$.
For any set $L \subset B^*$, define 
$\widehat{L}:=\{ \hat{w} | w \in L\}$. 
Define
$$
\nf_G := \nf_K \hq.
$$
Since $Q \cong G/K$ and the set $\hq \subset \hb^*$
bijects (via $q$) to $Q$, the language $\hq$ is a
set of coset representatives for $G/K$, 
and every element $g$ of $G$ can be written
uniquely in the form $g=kp$ for some $k \in K$ and $p \in \hq$;
that is, the set $\nf_G$ is a set of normal forms for $G$
over the finite \sym\ generating set $C$.

For each $g \in G$, let $y_g$ denote the
normal form of $g$ in $\nf_G$.  We also use the notation
$y_g=u_gt_g$ where $u_g \in \nf_K$
and $t_g \in \hq$.
Now $G$ has the presentation
\[G = \langle C \mid  R \cup 
\{ \hat{s} = u_{\hat{s}} \mid {s} \in {S}\} \cup
\{ \hat{b} a = u_{\hat{b}a\hat{b}^{-1}} \hat{b} \mid a \in A, \hat{b} \in \hat{B}\} \rangle. 
\]

Let $\ga$ be the Cayley graph of $G$ with respect to $C$,
and let $\vec E$ and $\vec P$ be the sets of directed edges
and directed paths in $\ga$.  
Since both $\nf_K$ and $\nf_Q$ are prefix-closed, 
the language $\nf_G$ is prefix-closed as well,
and so $\nf_G$ determines a maximal tree $\ttt$ in $\ga$.

\smallskip\eject

\noindent {\em Step 1: Stackable.}

Define a function $\sff: \nf_G \times C \rightarrow C^*$ by
$$
\sff(y_g, c) = \left\lbrace
\begin{array}{l l }
\sff_{K}(y_g,c) 
  & \mbox{if } c \in A \mbox{ and }  y_g  \in A^*
\\
\last(y_g)^{-1}u_{\last(y_g)c\last(y_g)^{-1}}\last(y_g)
  & \mbox{if }  c \in A \text{ and } y_g \notin A^*
\\
u_{c (\widehat{\sff_Q(q (g), q(c))})^{-1} }  \widehat{\sff_Q(q(g),q(c))} 
  & \mbox{if } c \in \hat{B}.
\end{array} \right.
$$
Also define $\ff:\vece \ra \vecp$ by
$\ff(e_{g,c}):=\path(y_g,\sff(y_g,c))$.


It follows immediately from 
the definition of $\ff$ and
the presentation of $G$ that
property (F1) of the definition of flow function holds.
To check property (F2d), suppose that $e=e_{g,c}$ is any
edge in the tree $\ttt$;
then either $y_gc$ or $y_{gc}c^{-1}$ lies in $\nf_G$.
If $c \in A$, then the definition of
$\nf_G$ implies that $y_g \in A^*$,
and either $y_gc$ or $y_{gc}c^{-1}$ lies in $\nf_K$.
Then $\sff_G(y_g,c)=\sff_K(y_g,c)=c$, and so $\ff(e)=e$.
If instead $c \in \hb$,
then either $u_gt_gc$ or $u_{gc}t_{gc}c^{-1}$ lies
in $\nf_G$.
Now $\overline{t_g} \in \nf_Q$,
$\overline{c} \in B$, and either
$\overline{t_g}\overline{c}$ or $\overline{t_{gc}}\overline{c}^{-1}$
lies in $\nf_Q$.
Hence property (F2d) for $\ff_Q$ implies that  
$\sff_Q(q(g),q(c))=\sff_Q(\overline{t_g},\overline{c})=\overline{c}$, 
and so $\widehat{\sff_Q(q(g),q(c))}=c$.  Note that
$u_{cc^{-1}}=u_{\gi}=\ew$.
Then again we have
$\sff(y_g,c)=c$ and $\ff(e)=e$.  Hence
property (F2d) holds for $\ff$.

Our procedure to check
 property (F2r) for $\ff$
will again (as in the proof of Theorem~\ref{thm:grprod})
make use of a function
$\psi:\vec E \ra \N^2$
that captures information 
from property (F2r) for the flow functions
$\ff_K$ and $\ff_Q$.
Let $\vec E_K=\{e_{k,a}^K \mid k \in K, a \in A\}$ 
be the set of directed edges
of the Cayley graph of $K$ over $A$, and define
$\redl_K:\vec E_K \ra \N$ by 
$$
\redl_K (e_{k, a}^K) = \mbox{maximal length of a descending chain } 
  e_{k,a}^K >_{\ff_K} e' >_{\ff_K} e'' \cdots.
$$ 
Similarly let $\vec E_Q=\{e_{q,b}^Q \mid q \in Q, b \in B\}$
 be the set of directed edges
of the Cayley graph of $Q$ over $B$, and let
$\redl_Q(e_{q,b}^Q)$  be the maximum length
of a  descending chain starting at $e_{q,b}^Q$
for the well-founded strict partial ordering $>_{\ff_Q}$.
Now define $\psi : \vec{E} \rightarrow \N^2$
on the directed edges of the Cayley graph of 
$G$ over $C$ by
$$
\psi (e_{g, c}) = \left\lbrace 
\begin{array}{ll}
(1,\redl_K(e_{g,c}^K) )
  & \mbox{if } c \in A \mbox{ and }  y_g \in A^*
\\
(2, \ell (t_g)) 
  &\mbox{if } c \in A \mbox{ and }  y_g \notin A^*
\\
(3, \redl_Q(e_{q(g),q(c)}^Q)) 
  & \mbox{if }  c \in \hat{B}
\end{array}\right.
$$
for all $g \in G$ and $c \in C$.
To prove property (F2r), it now suffices
to show  that  $e'<_\ff e$ implies
$\psi(e') <_{\N^2} \psi(e)$ (where $<_{\N^2}$
is the lexicographic ordering) whenever $e',e$ do not
lie in $\ttt$ and $e'$ is a directed edge on the
path $\ff(e)$.

To that end, let $e=e_{g,c}$ be any directed edge in $\ga$
that does not lie in the tree $\ttt$.

\noindent {\em Case 1: Suppose that $c \in A$ and $y_g \in A^*$.}  

Each edge $e'$ in the path $\ff(e)=\path(y_g,\sff_K(y_g,c))$ has the 
form $e'=e_{g',c'}$ with $y_{g'} \in A^*$ and
$c' \in A$.  Moreover, the edge $e_{g',c'}^K$ lies 
in the path $\ff_K(e_{g,c}^K)$,
and so $e_{g',c'}^K <_{\ff_K} e_{g,c}^K$
and $\redl_K(e_{g',c'}^K)) \le \redl_K(e_{g,c}^K)-1$.
Now $\psi(e')=(1,\redl_K(e_{g',c'}^K))
<_{\N^2} (1,\redl_K(e_{g,c}^K))=\psi(e)$.

\noindent {\em Case 2: Suppose that $c \in A$ and $y_g \notin A^*$.} 

The first edge in the path 
$\ff(e)=\path(y_g,\last(y_g)^{-1}u_{\last(y_g)c\last(y_g)^{-1}}\last(y_g))$ is
$e_{y_g,\last(y_g)^{-1}}$;
since this is also
the last edge of $\path(\gi,y_g)$,
this first edge lies in the tree $\ttt$.
Similarly, the last edge in $\ff(e)$ is
$e_{g\last(y_g)^{-1}u_{\last(y_g)c\last(y_g)^{-1}},\last(y_g)}$,
and the terminal vertex of this edge is
$gc =_G 
[u_gt_gct_g^{-1}]t_g=kt_g$ where $k=u_gt_gct_g^{-1}$
lies in $K$.
Then the normal form of 
$gc$ 
can be written as $u_kt_g$, 
ending in the letter $\last(y_g)$,
and so this last edge also lies in $\ttt$.

Now suppose that $e'$ is any edge in the
path $\ff(e)$ that does not lie in $\ttt$.  
Then $e'=e_{g',c'}$ satisfies 
$c' \in A$ and $g' \in G$ has a normal form
$u_{g'}t_{g'}$ with $t_g=t_{g'}\last(y_g)$.
In this case either $t_{g'}=\ew$ and 
$\psi(e')=(1,\redl_K(e_{g',c'}^K))$,
or $t_{g'} \neq \ew$ and
$\psi(e')=(2,l(t_{g'}))=(2,l(t_g)-1)$.
In either subcase, $\psi(e') <_{\N^2} (2,l(t_g))=\psi(e)$.

\noindent {\em Case 3: Suppose that $c \in \hb$.}

In this last case $\psi(e) = (3,\redl_Q(e_{q(g), q(c)}^Q))$.
The path
$\ff(e)=\path(y_g,u_{c (\widehat{\sff_Q(q (g), q(c))})^{-1} }
   \widehat{\sff_Q(q(g),q(c)))}$
is the concatenation of two paths 
$\rho_1 = \path(y_g,u_{c (\widehat{\sff_Q(q (g), q(c))})^{-1} })$
and
$\rho_2= \path(y_{g'},\widehat{\sff_Q(q(g),q(c))})$
where $g':= g u_{c (\widehat{\sff_Q(q(t_g), q(c))})^{-1}}$ is the
vertex at the terminus of $\rho_1$ and the
start of $\rho_2$.

For any edge $e'$ in the first subpath $\rho_1$, the
label on the edge $e'$ is an element of
$A$, and so 
$\psi(e')=(m,n)$ with $m<3$.
Hence $\psi(e') <_{\N^2} \psi(e)$.

To analyze the situation for an edge $e'$ of
the path $\rho_2$, we first note that 
the initial vertex $g'$ of $\rho_2$
satisfies $q(g')=q(g)$.
Then $e'$
has the form $e'=e_{u'\widehat{t'},\widehat{c'}}$ for some $u' \in \nf_K$ and
some edge $r(e'):=e_{t',c'}^Q$ of the path $\ff_Q(e_{q(g),q(c)})$.
Then
$\redl_Q(r(e')) < \redl_Q(e_{q(g),q(c)})$,
and
$\psi(e')=( 3,\redl_Q(e_{q(u'\widehat{t'}),q(\widehat{c'})}) )
=( 3,\redl_Q(e_{t',c'}) ) <_{\N^2}
( 3,\redl_Q(e_{q(g),q(c)}) ) = \psi(e)$.

\smallskip

This completes the proof of property (F2r) for $\ff$, 
and so $\ff$ is a flow function. Let $k_K$ and $k_Q$ be
the bounds on the flow functions $\ff_K$ and $\ff_Q$.
Let $M:=\max\{l(u_{dcd^{-1}}) \mid d \in \hb$ and $c \in A\}$,
and $m:=\max\{l(u_{cz}) \mid c \in \hb$ and $z$ is in the (finite)
image of $\sff_Q\}$.
Then $\max\{k_K,2+M,k_Q+m\}$ is a bound for the flow function $\ff$.

\smallskip

\noindent {\em Step 2:  Autostackable.}

In this step we assume that the groups $K$ and $Q$ are autostackable,
and in particular that the sets $graph(\sff_K)$ and $graph(\sff_Q)$
are \sr.  
We partition the finite image sets
$im(\sff_K) \subset A^*$ and $im(\sff_Q) \subset B^*$
as follows.  For each $c \in A$, let
$U_c :=\{\sff(y,c) \mid y \in \nf_K\}$, and
for each $c \in \hb$, let
$V_c := \{\sff(y,\overline{c}) \mid y \in \nf_Q\}$;  that
is, $U_c$ is the finite set of labels on paths
obtained from the flow function $\sff_K$ action on
edges with label $c$, and similarly for $V_c$.

The stacking function associated to the bounded flow
function $\ff$ for $G$ from Step 1 of this proof 
is the function $\sff$ defined in Step 1.
Using the piecewise definition of $\sff$, we have 
\begin{eqnarray*}
graph(\sff) &=& \hspace{.16in} (\cup_{c \in A,~z \in U_c}
       ~L_{c,z}   \times \{c\}        \times \{z\})\\
&& \medcup (\cup_{c \in A,~b \in \hb} 
       ~L_{c,b}'  \times \{c\}        \times \{b^{-1}u_{bcb^{-1}}b\}) \\
&& \medcup (\cup_{c \in \hb,~z \in V_c} 
       ~L_{c,z}'' \times \{c\}        \times \{u_{c\hat{z}}\hat{z}\}) ,
\end{eqnarray*}
where 

\hspace{.1in} $L_{c,z}=\{y_g \in \nf_G \mid y_g \in A^*$ and $\sff_K(y_g,c)=z\}$,

\hspace{.1in} $L_{c,b}'=\{y_g \in \nf_G \mid y_g \notin A^*$ and $\last(y_g) = b\}$,~~and

\hspace{.1in} $L_{c,z}''=\{y_g \in \nf_G \mid \sff_Q(q(y_g),q(c))=z\}$.

The first language $L_{c,z}$ is the set
${\mathsf{proj}}_1(graph(\sff_K) \cap (A^* \times \{c\} \times \{z\}))$.
Synchronous regularity 
of both sets in the intersection, along with
Lemma~\ref{lem:proj}, shows that $L_{c,z}$ is regular.
The second language is $L_{c,b}'=\nf_G \cap C^*b$.
Now $\nf_K={\mathsf{proj}}_1(graph(\sff_K))$ and
$\hq$ is a homomorphic image (via 
the map $\hat{~}$) of ${\mathsf{proj}}_1(graph(\sff_Q))$,
so these languages, as well as their concatenation $\nf_G$,
are regular, and therefore so is $L_{c,b}'$.
Finally, $L_{c,z}''$ is the concatenation
$L_{c,z}''=\nf_K(\widehat{{\mathsf{proj}}_1(graph(\sff_Q) \cap 
           (B^* \times \{q(c)\} \times \{z\}))})$;
similar arguments show that $L_{c,z}''$ is also regular.
Using the closure of \sr\ languages under finite
unions and
Lemma~\ref{lem:product} now shows that $graph(\sff)$ is
also regular.
Thus $G$ is autostackable over $C$.

\smallskip

\noindent {\em Step 3:  Algorithmically stackable.}

Again the proof in this step is nearly identical to
the proof of Step 2.
\end{proof}



\subsection{Finite index supergroups}\label{subsec:finext}


$~$

\vspace{.1in}

In this section we show that a group containing a
stackable, algorithmically stackable or
autostackable finite index subgroup must also 
have the same property.
While there are many similarities with the 
result and proof in  Section~\ref{sec:ext},
and so we do not include all of the details
of the proof,
the argument in the present section 
requires a different flow function because
we do not require the subgroup to be normal.

\begin{theorem}\label{thm:finext}
Let  $H$ be an
autostackable
[respectively, stackable, \astkbl] group on a
finite \sym\ generating set $A$,
let $G$ be a group containing $H$ as a subgroup of
finite index, and let $S \subseteq G$ be
a set of coset representatives for $G/H$ containing $\gi$.
Then the group $G$ with the generating set 
$A \cup (S \setminus \{\gi\})^{\pm 1}$
is also autostackable
[respectively, stackable, \astkbl]. 
\end{theorem}

\begin{proof}
Let $\nf_H$, $\ff_H$, and $\sff_H$ be the normal form
set, bounded flow function, and associated stacking
map for $H$ over $A$, and let
$H = \langle A | R \rangle$
be the finite presentation obtained from 
this flow function. 
Let $B := (S \setminus \{\gi\})^{\pm 1}$ and 
let $C := A \cup B$.  

Since $S$ is a transversal, the set
$$\nf_G := \nf_H \cup \{ ut | u \in \nf_H, t \in S \setminus \{\gi\} \}$$
is a set of normal forms for $G$ over $C$.
Moreover, prefix-closure of the language $\nf_H$ 
implies that $\nf_G$ is also prefix-closed.

Given any $g \in G$, we write the
normal form from $\nf_G$ for $g$ as $y_g=u_gt_g$ 
where $u_g \in \nf_H$ and 
$t_g \in \{\ew\} \cup (S \setminus \{\gi\})$.
Using this notation, the group $G$ is presented by\\
\-\ 
\hspace{.53in}
$G 
= \langle C \mid R \cup 
  \{ x = u_x t_x \mid x \in B \setminus S \} \cup 
  \{xy = u_{xy} t_{xy} \mid x \in B, y \in C \}\rangle.
$\\
Let $\ga$ be the Cayley graph for $G$ with respect to $C$,
with sets $\vec E$ and $\vec P$ of directed edges and paths,
and let $\ttt$ be the maximal tree in $\ga$
determined by the set $\nf_G$ of normal forms.

\smallskip

\noindent {\em Step 1:  Stackable.}

Define a function $\sff : \nf_G \times C \ra C^*$ by 
\begin{center}
$\sff(y_g,c) = \left\lbrace
\begin{array}{ll}
\sff_H(y_g,c) 
  & \mbox{if } c \in A \mbox{ and } y_g \in A^*
\\
y_c 
  & \mbox{if } c \in B \mbox{ and } y_g \in A^*
 \\
\last(y_g)^{-1}  y_{\last(y_g)c}           
  & \mbox{if } c \in C \mbox{ and } y_g \notin A^*
\end{array} \right.$
\end{center}
Also as usual define $\ff:\vece \ra \vecp$ by
$\ff(e_{g,c}):=\path(y_g,\sff(y_g,c))$.

It is again immediate from the definition of $\ff$ that
property (F1) of the definition of flow function holds.
To check property (F2d), suppose that $e=e_{g,c}$ is any
edge in the tree $\ttt$.  
Then either
$y_gc$ or $y_{gc}c^{-1}$ lies in $\nf_G$.
If $c \in A$, then $y_g \in A^*$
and either $y_gc$ or $y_{gc}c^{-1}$ is in $\nf_H$.
Then property (F2d) for $\sff_H$
implies that $\sff_G(y_g,c)=\sff_H(y_g,c)=c$.  
If $c \in B$ and
$y_g \in A^*$, then 
$y_g \neq y_{gc} c^{-1}$,
and consequently
$y_gc \in \nf_G$.  Then $c \in S \setminus \{\gi\}$
and $\sff_G(y_g,c)=y_c=c$.
Finally, if $c \in B$ and $y_g \notin A^*$,
then $y_gc \notin \nf_G$, and so $y_{gc}c^{-1} \in N_G$.
In this case 
$\last(y_g)=t_g=c^{-1} \in S \setminus \{\gi\}$,
and $\sff_G(y_g,c)=\last(y_g)^{-1}  y_{\last(y_g)c}
=cy_1=c$.   
Then in all cases we have $\ff(e)=e$;
therefore (F2d) holds for $\ff$.

Next define $\psi: \vec E \rightarrow \N^2$ by
$$
\psi (e_{g, c}) = \left\lbrace 
\begin{array}{ll}
(0,\redl_H (e_{g,c}^H)) 
  & \mbox{if } c \in A \mbox{ and } y_g \in A^* 
\\
(1,0) 
  & \mbox{if } c \in B \mbox{ and } y_g \in A^*
\\
(1, 1) 
  & \mbox{if } c \in C \mbox{ and } y_g \notin A^*
 \\
\end{array}\right.
$$
where (as in the proof of Theorem~\ref{thm:extn}),
for an edge $e_{g,c}^H$ in the Cayley graph of
$H$ over $A$, $\redl_H (e_{g,c}^H) \in \N$ is the
maximal length of a descending chain 
$e_{g,c}^H >_{\ff_{H}} e' >_{\ff_{H}} e'' \cdots$.
As usual, let $<_{\N^2}$ be the lexicographic order on
$\N^2$.

Let $e=e_{g,c}$ be any directed edge in $\ga$
whose underlying undirected edge is not in $\ttt$.

\noindent {\em Case 1: Suppose that $c \in A$ and $y_g \in A^*$.}  

The proof in this case is similar to Case 1 of Step 1
in the proof of Theorem~\ref{thm:extn}.

\noindent {\em Case 2: Suppose that $c \in B$ and $y_g \in A^*$.}  

In this case $\psi(e)=(1,0)$, and the path $\ff(e)$ is
labeled by the word $y_c$.
Since the edge $e$ is not in $\ttt$, the word
$y_gc \notin \nf_G$, and so
$c \in B \setminus S$ and $y_c=u_ct_c$ for some
$u_c \in \nf_H$ and $t_c \in S \setminus \{\gi\}$.
Each edge $e'$ in the subpath $\path(y_g,u_c)$ of $\ff(e)$
satisfies $\psi(e')=(0,n)$ for some $n \in \N$,
and so $\psi(e') <_{\N^2} \psi(e)$.
The final edge $e_{gu_c,t_c}$ of
$\ff(e)$ lies in the tree $\ttt$.

\noindent {\em Case 3: Suppose that $c \in C$ and $y_g \notin A^*$.}

In this case $\psi(e)=(1,1)$, and the path $\ff(e)$ is
labeled by the word $\last(y_g)^{-1}  y_{\last(y_g)c}$.
The first edge $e_{g,\last(y_g)^{-1}}$ lies in $\ttt$.
Since $\last(y_g)=t_g$,
the next subpath of $\ff(e)$ is 
$\path(u_g,u_{t_gc})$.
Any edge $e'$ in this subpath satisfies
$\psi(e')=(0,n)$ for some $n \in \N$; 
hence $\psi(e') <_{\N^2} \psi(e)$.
The remainder of the path $\ff(e)$ is
the edge $e_{u_gu_{t_gc},t_{t_gc}}$,
which lies in the tree $\ttt$.

\smallskip

We now have that $\psi(e') <_{\N^2} \psi(e)$
whenever $e' <_{\ff} e$, and so property (F2r)
holds and $\ff$ is a flow function.
Finally, since $A$, $B$, $C$, and the image
$im(\sff_H)$ of the stacking map for $H$ 
are finite sets, the flow function $\ff$ is
bounded.  Therefore $G$ is stackable over $C$.

\smallskip

\noindent{\em Step 2:  Autostackable and algorithmically stackable.}

The map $\sff$ from Step 1 is the stacking map for
the flow function $\ff$ for $G$, and the graph
of this function can be decomposed as a finite union of sets
\begin{eqnarray*}
graph(\sff) &=& \hspace{.16in}(\cup_{c \in A,~z \in \{\sff_H(y,c) \mid y \in \nf_H\}}
       ~{\mathsf{proj}}_1(graph(\sff_H) \cap (A^* \times \{c\} \times \{z\}))
          \times \{c\}        \times \{z\})\\
&& \medcup (\cup_{c \in B} 
       ~\nf_H  \times \{c\}        \times \{y_c\}) \\
&& \medcup (\cup_{c \in C,~s \in S \setminus \{\gi\}} 
       ~\nf_Hs \times \{c\}        \times \{s^{-1}y_{sc}\}).
\end{eqnarray*}
With the added assumption that $graph(\sff_H)$
is either \sr\ or computable, then
$graph(\sff)$ satisfies the same property.
\end{proof}


\section{Homological finiteness} \label{sec:stallings}


The purpose of this section is to investigate the homological 
properties of autostackable groups by studying 
Stallings'~\cite{Stallings} non-$FP_3$ group
$$
G := \langle a,b,c,d,s \mid [a,c] = [a,d] = [b,c] = [b,d] = 1, 
[s, ab^{-1}] = [s, ac^{-1}] = [s, ad^{-1}] = 1\rangle
$$
with respect to the generating set 
$A := \{ a^{\pm 1}, b^{\pm 1}, c^{\pm 1}, d^{\pm 1}, s^{\pm 1} \}$.  

The group $G$ is an HNN extension, with stable
letter $s$, of the 
direct product of two free groups of
rank 2, 
$$H = F_2 \times F_2 =
\langle a,b,c,d \mid [a,c] = [a,d] = [b,c] = [b,d] = 1\rangle
$$
generated by the subset 
$Z := \{ a^{\pm 1}, b^{\pm 1}, c^{\pm 1}, d^{\pm 1}\}$.
Since the relations of this presentation have zero
exponent sum as words over $Z$, given any element $h \in H$,
there is a unique number $\es(h)$ such that
every word over $Z$ representing $h$ has exponent
sum $\es(h)$.
Let $N$ denote the subgroup of $H$
of elements of zero exponent sum. Then $N$ is
a normal subgroup (as conjugation preserves
exponent sum) of $H$, and $N$ is
generated by
$ab^{-1}$, $ac^{-1}$, and $ad^{-1}$.  In the
HNN extension $G$, conjugation
by the stable letter $s$ is the identity
map on $N$.  (See for example~\cite{elderherm}
for more details.)

The set $\{a^i \mid i \in \Z\}$ is 
both a left transversal and right transversal of $N$ in $H$.
Let 
$$
\nf_H := \{uv \mid u \in \{a^{\pm 1}, b^{\pm 1}\}^*
\text{ and } v \in \{c^{\pm 1}, d^{\pm 1}\}^* 
\text{ are freely reduced} \};
$$
this is a set of normal forms 
for $H$ over $Z$.  Then 
\[
\begin{array}{rl}
\nf_G  := \{ ws^{\epsilon_1} a^{i_1} s^{\epsilon_2}a^{i_2} 
\cdots s^{\epsilon_n}a^{i_n} \mid &
 w \in \nf_H,~n\geq 0,~\epsilon_j \in \{\pm 1\} \text{ and } 
i_j \in \mathbb{Z} \text{ for all } j, \\
&\mbox{ and whenever } i_j = 0,\mbox{ then }
\epsilon_{j} = \epsilon_{j+1} \}
\end{array}
\]
is a set of normal forms for $G$ over $A$
(using normal forms for 
HNN extensions; for example, 
see~\cite[Theorem~IV.2.1]{LyndonSchupp}).
In fact, the set $\nf_G$ is the set of irreducible
words of the (infinite) complete rewriting system
\[
\begin{array}{rl}\label{rs:stallings}
R_G  = & \{xx^{-1} \ra \ew \mid x \in A\}~\cup~\{yx \ra xy \mid 
      x \in \{a^{\pm 1}, b^{\pm 1}\},~y \in \{c^{\pm 1}, d^{\pm 1}\}\} \\
&  \cup~\{s^\epsilon a^i y^\eta \ra a^{-\eta}y^\eta s^\epsilon a^{\eta+i} \mid
      \epsilon,\eta \in \{\pm 1\},~i \in \Z,~y \in \{c, d\}\} \\
&  \cup~\{s^\epsilon a^i b^\eta \ra a^{i}b^\eta a^{-\eta-i} s^\epsilon a^{\eta+i} \mid
      \epsilon,\eta \in \{\pm 1\},~i \in \Z\}\}
\end{array}
\]
for $G$ over $A$.
We also note that the language $\nf_G$ is prefix-closed,
and so determines a maximal tree in the Cayley
graph for $G$ over $A$.

\begin{theorem}\label{thm:stallings}
Stallings' non-$FP_3$ group $G$ is autostackable.
\end{theorem}

\begin{proof}
Let $\nf_G$ be the normal form set for $G$
over the generating set $A$ described above, and denote the normal form
for any element $g \in G$ by $y_g$.
Let $\ga$ be the Cayley graph of $G$ over $A$, with
sets $\vece$ and $\vec P$ of directed
edges and paths, respectively, and let
$\ttt$ be the tree in $\ga$ corresponding to the set $\nf_G$ of
normal forms.

\smallskip\eject

\noindent {\em Step 1:  Stackable.}

Define a function
$\sff : \nf_G \times A \rightarrow A^*$ by
\[
\sff(y_g,x) := \begin{cases}
x & \mbox{if either } y_g x \in \nf_G \text{ or } y_{gx}x^{-1} \in \nf_G
\\
\last(y_g)^{-1}x\last(y_g) & 
       \mbox{if } 
        x \in \{a^{\pm 1},~b^{\pm 1}\},~y_g \in Z^*, \text{ and } 
       \last(y_g) \in \{c^{\pm 1}, d^{\pm 1}\}
\\
\last(y_g)^{-1}x\last(y_g) & 
       \mbox{if }
       x \in \{c^{\pm 1}, d^{\pm 1}\},~y_g \notin Z^*, \text{ and }
       \last(y_g) \in \{a^{\pm 1}\}
\\
c^{-\eta}xc^{\eta} & 
       \mbox{if } 
       x \in \{b^{\pm 1}\},~y_g \notin Z^*, \eta \in \{\pm 1\}, 
       \text{ and }
       \last(y_g) = a^\eta
\\
\last(y_g)^{-1}xa^{-\eta}\last(y_g) a^\eta & 
       \mbox{if } 
       x \in \{b^\eta, c^{\eta}, d^{\eta}\} \text{ with } \eta \in \{\pm 1\}, 
       \text{ and }
       \last(y_g) \in \{s^{\pm 1}\}
\end{cases}
\]
for all $y_g \in \nf_G$ and $x \in A$. 
In all of the cases that do not appear explicitly,
namely when either [$x \in \{s^{\pm 1}\}$]; 
[$x \in \{a^{\pm 1},~b^{\pm 1}\}$,~$y_g \in Z^*$, and 
       $\last(y_g) \in \{1, a^{\pm 1}, b^{\pm 1}\}$];
[$x \in \{c^{\pm 1}, d^{\pm 1}\}$ and $y_g \in Z^*$];
or [$x \in \{a^{\pm 1}\}$ and $y_g \notin Z^*$],
it follows from the definition of $\nf_G$ 
that either
$y_g x$ or $y_{gx} x^{-1}$ lies in $\nf_G$.  Moreover,
one can also check that
the five cases in the definition of $\sff$ are disjoint;
that is, the function $\sff$ is well-defined.

Let $\ff:\vece \ra \vecp$ denote the function
$\ff(\ega):=\path(y_g,\sff(y_g,a))$.
It follows immediately from the definition of $\ff$ and
the presentation of $G$ that
properties (F1) and (F2d) of the definition of flow function hold
for $\ff$.

In order to prove that property (F2r) holds for $\ff$,
we utilize the following function $\psi: \vec E \ra \N^3$.  Define
$\psi(e_{g,x}):=(0,0,0)$ if $e_{g,x}$ lies in the tree $\ttt$,
and if $e_{g,x}$ does not lie in $\ttt$, let
\[
\psi (e_{g, x}) := \left\lbrace \begin{array}{ll}
(0, 0, l(\suf_{\{c^{\pm 1},d^{\pm 1}\}}(y_g))) & 
                \mbox{if } x \in \{a^{\pm 1}, b^{\pm 1}\} \mbox{ and } y_g \in Z^* \\
(n_s(y_g), l(\suf_{\{a^{\pm 1}\}}(y_g)), 0) & \mbox{if }
x \in \{c^{\pm 1}, d^{\pm 1}\} \mbox{ and } y_g \notin Z^*\\
(n_s(y_g), l(\suf_{\{a^{\pm 1}\}}(y_g)), 1) & \mbox{if }
x \in \{b^{\pm 1}\} \mbox{ and } y_g \notin Z^*,
\end{array}\right.
\]
where $n_s(y_g)$ denotes the number of occurrences
of $s^{\pm 1}$ in the word $y_g$.
Let $<_{\N^3}$ denote the lexicographical ordering
on $\N^3$ obtained from the standard ordering on $\N$,
a well-founded strict partial ordering.
To prove (F2r), then, it suffices to show that whenever
$e' <_\ff e$, then $\psi(e') <_{\N^3} \psi(e)$.

Let $e=e_{g,x} \in \vece$ be any edge whose
underlying undirected edge does not lie in $\ttt$.

\noindent {\em Case 1:
Suppose that $x \in \{ a^{\pm 1}, b^{\pm 1} \}$, $y_g \in Z^*$, 
and $\last(y_g) \in \{c^{\pm 1}, d^{\pm 1}\}$.}

In this case
$\psi (e) = (0, 0, l(\suf_{\{c^{\pm 1},d^{\pm 1}\}}(y_g)))$.  
The path $\ff(e)$ contains three directed
edges: $e_1:=e_{g, \last(y_g)^{-1}}$, 
$e_2:=e_{g\last(y_g)^{-1}, x}$ and
$e_3:=e_{g\last(y_g)^{-1}x, \last(y_g)}$.  The 
edge $e_1$ lies in the tree $\ttt$, as 
$y_{g\last(y_g)^{-1}}\last(y_g)= y_g$. The edge $e_2$
either lies in $\ttt$, or else its image under $\psi$ is
$\psi(e_2) = 
(0, 0, l(\suf_{\{c^{\pm 1},d^{\pm 1}\}}(y_{g\last(y_g)^{-1}}))$. 
Since $y_{g\last(y_g)^{-1}}$ is the prefix of $y_g$
consisting of all but the last letter $\last(y_g)$
(which lies in $\{c^{\pm 1}, d^{\pm 1}\}$),
we have 
$l(\suf_{\{c^{\pm 1},d^{\pm 1}\}}(y_{g\last(y_g)^{-1}}))=
l(\suf_{\{c^{\pm 1},d^{\pm 1}\}}(y_g))-1$ and
$\psi(e_2) <_{\N^3} \psi(e)$.
To analyze the edge $e_3$, we  
decompose $y_g = uv\last(y_g)$ where $u \in \{ a^{\pm 1}, b^{\pm 1} \}^*$ 
and $v\last(y_g) \in \{ c^{\pm 1}, d^{\pm 1}\}^+$ are reduced words.
Now  the normal forms satisfy
$y_{g\last(y_g)^{-1} x}\last(y_g) = y_{ux}v\last(y_g)=
y_{g\last(y_g)^{-1} x\last(y_g)}$, and
so the edge $e_3$ also lies in the tree $\ttt$.

\noindent {\em Case 2: Suppose that 
$x \in \{ c^{\pm 1}, d^{\pm 1} \}$, $y_g \notin Z^*$,
and $\last(y_g) \in \{a^{\pm 1}\}$.}

In this case we have $\psi(e) = (n_s(y_g), l(\suf_{\{a^{\pm 1}\}}(y_g)), 0)$,
and there are three directed edges in the path $\ff(e)$.
Similar to case 1, the first of these edges,
$e_{g,\last(y_g)^{-1}}$, lies in the tree $\ttt$.
The second edge,
$e_2:=e_{g\last(y_g)^{-1}, x}$ has $\psi$ function value of 
$\psi(e_2) = 
(n_s(y_g), l(\suf_{\{a^{\pm 1}\}}(y_g))-1, 0) <_{\N^3} \psi (e)$. 
The third edge is 
$e_3:=e_{g\last(y_g)^{-1}x, \last(y_g)}$. 
Applying the rewriting system $R_G$
above shows that the normal form of the word $y_g\last(y_g)^{-1}x$ 
again contains $n_s(y_g)>0$
appearances of $s^{\pm 1}$, and therefore the edge 
$e_3$
also lies in $\ttt$.

\noindent {\em Case 3: Suppose that
$x \in \{b^{\pm 1}\}$, $y_g \notin Z^*$, 
$\eta \in \{\pm 1\}$,
and $\last(y_g)=a^\eta$.}

Now 
$\psi (e) = (n_s(y_g), l(\suf_{\{a^{\pm 1}\}}(y_g)), 1)$, and
the path $\ff(e)$ contains three directed edges:
$e_1:=e_{g, c^{-\eta}}$, $e_2:=e_{gc^{-\eta}, x}$ and
$e_3:=e_{gc^{-\eta}x, c^{\eta}}$.  Unlike the previous cases, 
none of these edges lie in $\ttt$.
The edge $e_1$ satisfies
$\psi(e_1) = (n_s(y_g), l(\suf_{\{a^{\pm 1}\}}(y_g)), 0)
<_{\N^3} \psi(e)$.

For the analysis of the other two edges, we first use the 
definition of the set $\nf_G$ to write 
out the normal form 
$y_g=ws^{\epsilon_1}a^{i_1} \cdots s^{\epsilon_n}a^{i_n}$
where $w \in \nf_H$, $n>0$, $\epsilon_j \in \{\pm 1\}$
and $i_j \in \Z$ for all $j$, and $i_n/|i_n| = \eta$.
Note that with this notation, $\psi (e) = (n, |i_n|, 1)$.

The normal form for $gc^{-\eta}$ is 
$y_{gc^{-\eta}} = 
y_{wc^{-\eta}a^{\eta}}s^{\epsilon_1}a^{i_1} \cdots 
s^{\epsilon_n}a^{i_n-\eta}$.
Therefore the image of
$e_2$ under $\psi$ satisfies
$\psi (e_2) = (n, |i_n|-1, 1)
<_{\N^3} \psi (e)$.  

Writing $x=b^\beta$ with $\beta \in \{\pm 1\}$, then
the value of $\psi(e_3)$ depends upon the sign
of the product $\beta \cdot \eta$.  
If $\beta \cdot \eta = 1$, 
then the normal form of the element $gc^{-\eta}x$ of $G$ is
$y_{h}s^{\epsilon_1}a^{i_1} \cdots s^{\epsilon_n}a^{i_n}$ where  
$h =_H wa^{i_1 + \cdots +i_n}c^{-\eta}xa^{-(i_1 + \cdots +i_n)}$.
In this subcase, then,
$\psi(e_3) = (n, |i_n|, 0) <_{\N^3} \psi(e)$.
On the other hand if $\beta \cdot \eta = -1$,
then the normal form of $gc^{-\eta} x$ is
$y_{h'}s^{\epsilon_1}a^{i_1} \cdots s^{\epsilon_n}a^{i_n-2\eta}$, 
where 
$h'=_H wa^{i_1 + \cdots +i_n}c^{-\eta}xa^{-(i_1 + \cdots +i_n)+2\eta}$.  
Since $\eta = i_n/|i_n|$, then
$|i_n - 2\eta| \le |i_n|$.
Therefore in this subcase we have
$\psi (e_3)  = (n, |i_n-2\eta|, 0)  <_{\N^3} \psi(e)$.

\noindent {\em Case 4: Suppose that
 $x \in \{ b^{\eta}, c^{\eta}, d^{\eta} \}$ with 
$\eta \in \{\pm 1\}$, and 
$\last(y_g) \in \{s^{\pm 1}\}$.}

In this case 
$\psi(e)=(n_s(y_g),0,m)$ 
(with $m \in \{0,1\}$ depending on $x$) and
$\ff(e)$ contains five 
directed edges: 
$e_1:=e_{g, \last(y_g)^{-1}}$, 
$e_2:=e_{g\last(y_g)^{-1}, x}$, 
$e_3:=e_{g\last(y_g)^{-1}x, a^{-\eta}}$, 
$e_4:=e_{g\last(y_g)^{-1}xa^{-\eta}, \last(g)}$ and 
$e_5:=e_{g\last(y_g)^{-1}xa^{-\eta}\last(g), a^{\eta}}$.
Edges $e_1$ and $e_4$ both lie in the tree $\ttt$, since
every edge labeled by $s^{\pm 1}$ lies in this tree. 
The initial vertex of the edge $e_5$ is the element
$g'=_G g\last(y_g)^{-1}xa^{-\eta}\last(g)$ of $G$;
since $g \notin H$ and $\last(y_g)^{-1}xa^{-\eta}\last(g) \in H$,
then $g' \notin H$ and $y_{g'} \notin Z^*$.
Hence the edge $e_5$ also lies in $\ttt$, 
as any edge labeled by
the letter $a^{\pm 1}$ with initial vertex having a
normal form outside of $Z^*$ lies in $\ttt$.
 
If $n_s(y_g)>1$, then
$\psi(e_2) = (n_s(y_g)-1, l(\suf_{\{a^{\pm 1}\}}(y_{g\last(y_g)^{-1}})), m)$ 
and thus 
$\psi(e_2) <_{\N^3} \psi(e)$. 
Moreover when $n_s(y_g) > 1$, the argument 
above showing that $e_5$ lies in $\ttt$ applies to show 
that $e_3$ lies in $\ttt$ as well.
On the other hand, if $n_s(y_g) = 1$, then the image
via $\psi$ for both 
$e_2$ and $e_3$ has
the form $(0,0, l(\suf_{\{c^{\pm 1},d^{\pm 1}\}}(y))$ for 
a word $y \in \nf_H$, or else is (0,0,0).
In this case, we also have both 
$\psi(e_2) <_{\N^3} \psi(e)$ 
and $\psi(e_3) <_{\N^3} \psi(e)$. 

\smallskip

These four cases show that for any 
directed edge $e'$ that is in $\ff(e)$ but not in $\ttt$, 
the inequality
$\psi(e') <_{\N^3} \psi(e)$ holds.  Hence property (F2r) holds
for the function $\ff$, and $\ff$ is a flow function.
Moreover, this flow function is bounded, with bounding constant $k=5$.

\smallskip

\noindent {\em Step 2:  Autostackable.}

The function $\sff$ defined in Step 1 of this proof
is the stacking
function associated to the bounded flow function $\ff$.
It remains for us to show that the language
$graph(\sff)$ is \sr.  As in our earlier proofs,
we proceed by expressing $graph(\sff)$ as a union of other languages,
using the piecewise definition of $\sff$ from Step 1:
\begin{eqnarray*}
graph(\sff) &=& \hspace{.16in}(\cup_{x \in A}
 ~L_x \times \{x\} \times \{x\})\\
&& \medcup (\cup_{x \in \{a^{\pm 1},b^{\pm 1}\},~z \in \{c^{\pm 1},d^{\pm 1}\}} 
       ~L_{x,z} \times \{x\}
       \times \{z^{-1}xz\}) \\
&& \medcup (\cup_{x \in \{c^{\pm 1},d^{\pm 1}\},~z \in \{a^{\pm 1}\}} 
       ~L_{x,z}' \times \{x\}
       \times \{z^{-1}xz\}) \\
&& \medcup (\cup_{x \in \{b^{\pm 1}\},~\eta \in \{\pm 1\}} 
       ~L_{x,\eta} \times \{x\}
       \times \{c^{-\eta}xc^\eta\}) \\
&& \medcup (\cup_{\eta \in \{\pm 1\},~x \in \{b^{\eta},c^{\eta},d^{\eta}\},~z \in \{s^{\pm 1}\}} 
       ~L_{\eta,x,z} \times \{x\}
       \times \{z^{-1}xa^{-\eta}za^\eta\}),
\end{eqnarray*}
where 

$L_x=\{y_g \in \nf_G \mid$ either $y_g x \in \nf_G$ or $y_{gx}x^{-1} \in \nf_G\}$,

$L_{x,z}=\{y_g \in \nf_G \mid y_g \in Z^*$ and $\last(y_g) = z\}$,

$L_{x,z}'=\{y_g \in \nf_G \mid y_g \notin Z^*$ and $\last(y_g) = z\}$,

$L_{x,\eta}=\{y_g \in \nf_G \mid y_g \notin Z^*$ and $\last(y_g) = a^\eta\}$,~~and

$L_{\eta,x,z}=\{y_g \in \nf_G \mid \last(y_g)=z\}$.

\noindent Using Lemma~\ref{lem:product} and closure of \sr\ languages
under finite unions, it suffices to show that each of
the languages $L_x$, $L_{x,z}$, $L_{x,z}'$, $L_{x,\eta}$, and
$L_{\eta,x,z}$ is regular.




We start by considering 
the set $\nf_G$.
This is the set of irreducible words
for the rewriting system $R_G$, and so can be written as
$\nf_G = A^* \setminus A^*MA^*$ where

\hspace{.1in} $M = \{xx^{-1} \mid x \in A\}~\cup~
      \{c^{\pm 1}, d^{\pm 1}\}\{a^{\pm 1}, b^{\pm 1}\}
~\cup~ 
     s^{\pm 1}(a^* \cup (a^{-1})^*) \{b^{\pm 1}, c^{\pm 1}, d^{\pm 1}\}$.

\noindent Closure of the class of regular languages under
finite unions and concatenation shows that $M$ is regular;
closure under concatenation and complement then shows
that $\nf_G$ is regular.


The language $L_x$ can be expressed as 
$L_x = (\nf_G/x) \cup (\nf_G \cap A^*x^{-1})$.  
Applying Lemma~\ref{lem:peel}
and regularity of $\nf_G$, 
then $L_x$ is a regular language.

Note that 
$L_{x,z}= \nf_G \cap Z^* \cap A^*z$,
$L_{x,z}'=(\nf_G \cap A^*z) \setminus Z^*$, 
$L_{x,\eta}=(\nf_G \cap A^*a^\eta) \setminus Z^*$,
and 
$L_{\eta,x,z}=\nf_G \cap A^*z$, and so regularity of
these languages also follows
from regularity of the normal form set $\nf_G$.
\end{proof}

Theorem~\ref{thm:stallings} yields following.

\begin{corollary}\label{cor:autostknotfp3}
There is an autostackable group that does not satisfy
the homological finiteness condition $FP_3$. 
\end{corollary}

\begin{remark}  {\em
Recall from Section~\ref{sec:intro} that
Stallings' group $G$ cannot have a finite
complete rewriting system.  Earlier in
Section~\ref{sec:stallings} (on~p.~\pageref{rs:stallings}),
we gave an infinite complete rewriting system
for this group.  A consequence of Theorem~\ref{thm:stallings}
is that $G$ must also admit a \sr\ bounded 
\prs\ over the generating set $A$.  
For completeness, we record this system in
this remark; the \prs\ is:
\[
\begin{array}{rl}
R_G  = & \{zx^{-1}x \ra z \mid x \in A,~zx^{-1} \in \nf_G\} \\
& \cup~\{zyx \ra zxy \mid 
      x \in \{a^{\pm 1}, b^{\pm 1}\},~y \in \{c^{\pm 1}, d^{\pm 1}\},~zy \in \nf_G \cap Z^*\} \\
& \cup~\{zyx \ra zxy \mid 
      x \in \{c^{\pm 1}, d^{\pm 1}\},~y \in \{a^{\pm 1}\},~zy \in \nf_G \setminus Z^*\} \\
&  \cup~\{za^\eta x \ra  za^\eta c^{-\eta}xc^\eta \mid
      x \in \{b^{\pm 1}\},~\eta \in \{\pm 1\},~za^\eta \in \nf_G \setminus Z^*\} \\
&  \cup~\{zs^\epsilon x^\eta \ra  zx^{\eta}a^{-\eta} s^\epsilon a^\eta \mid
      x \in \{b,c,d\},~\epsilon,\eta \in \{\pm 1\},~zs^\epsilon \in \nf_G\}.
\end{array}
\]
}\end{remark}




\end{document}